\documentclass[12pt,reqno]{article}

\usepackage[usenames]{color}
\usepackage{amssymb}
\usepackage{graphicx}
\usepackage{amscd}
\usepackage{url}

\usepackage[colorlinks=true,
linkcolor=webgreen,
filecolor=webbrown,
citecolor=webgreen]{hyperref}

\definecolor{webgreen}{rgb}{0,.5,0}
\definecolor{webbrown}{rgb}{.6,0,0}

\usepackage{color}
\usepackage{fullpage}
\usepackage{float}

\usepackage{graphics,amsmath,amssymb}
\usepackage{amsthm}
\usepackage{amsfonts}
\usepackage{latexsym}
\usepackage{enumerate}

\newcommand{\Ft}[1]{\mathcal{F}_{#1}}
\newcommand{\Lt}[1]{\mathcal{L}_{#1}}

\theoremstyle{plain}
\newtheorem{theorem}{Theorem}

\newtheorem{proposition}[theorem]{Proposition}

\theoremstyle{remark}

\begin{document}

\begin{center}
\vskip 1cm{\large\bf Identities for the generalized Fibonacci polynomial}
\vskip 1cm
\large
Rigoberto Fl\'orez\\
Department of Mathematics and Computer Science\\
The Citadel\\
Charleston, SC \\
U.S.A.\\
{\tt rigo.florez@citadel.edu} \\
\ \\
Nathan  McAnally\\
Department of Mathematics and Computer Science\\
The Citadel\\
Charleston, SC \\
U.S.A.\\
{\tt nmcanall@citadel.edu}{\tt nmcanall@citadel.edu} \\
\ \\
Antara Mukherjee\\
Department of Mathematics and Computer Science\\
The Citadel\\
Charleston, SC \\
U.S.A.\\
{\tt antara.mukherjee@citadel.edu}
\end{center}

\centerline{\bf Abstract}
\noindent
A second order polynomial sequence is of Fibonacci type (Lucas type) if its Binet formula is similar
in structure to the Binet formula for the Fibonacci (Lucas) numbers. In this paper we generalize identities
from Fibonacci numbers and Lucas numbers  to Fibonacci type and Lucas type polynomials. A Fibonacci
type polynomial is equivalent to a Lucas type polynomial if they both satisfy the same recurrence
relations. Most of the identities provide relationships between two equivalent polynomials.
In particular, each type of identities in this paper relate the following polynomial sequences:
Fibonacci with Lucas, Pell with Pell-Lucas, Fermat with Fermat-Lucas, both types of
Chebyshev polynomials, Jacobsthal with Jacobsthal-Lucas and both types of Morgan-Voyce.

This paper is now published, see INTEGERS 18B (2018)\\ \url{http://math.colgate.edu/~integers/s18b2/s18b2.pdf}

\section {Introduction}

A second order polynomial sequence is of \emph{Fibonacci type} (\emph{Lucas type}) if its Binet formula
has a structure similar to that for Fibonacci (Lucas) numbers. In the literature these type of sequences are known as
\emph{Generalized Fibonacci Polynomial} (GFP). They are actually a natural generalization of the sequence
$F_{n}(x)= x F_{n-1}(x) + 1 F_{n-2}(x)$ where $F_{0}(x)= 0$, $F_{1}(x)= 1$ --called the
\emph{Fibonacci polynomial sequence}--. However, there is no unique generalization of this
sequence, one can refer to articles by several authors like Andr\'{e}-Jeannin \cite{Richard1, Richard2},
Bergum et al. \cite{bergum:hoggatt}, and Fl\'orez et al. \cite{florezHiguitaMukCharact,florezHiguitaMuk}, to see this.
In this paper we use the definition of the GFP given in \cite{florezHiguitaMukCharact,florezHiguitaMuk,HoggattLong}.
Fl\'orez et al. \cite{florezHiguitaMukCharact}, proved the strong divisibility property for the GFP,
but while working on it the authors needed some identities involving the polynomials.
When searching through the existing literature they realized that there were a limited
number of those identities, even just for Fibonacci polynomials.

The study of identities for Fibonacci polynomials and Lucas polynomials
have received less  attention than their counterparts for numerical sequences, even if many of these identities
can be proved easily.  A natural question to ask is: under what conditions is it possible to extend identities that
already exist for Fibonacci and Lucas numbers to the GFP?  We observe here that the identities involving
Fibonacci and Lucas numbers extend naturally to the GFP that satisfy closed formulas
similar to the Binet formulas satisfied by Fibonacci and Lucas numbers.

While this paper does not intend to prove new numerical identities, we have concluded with a
substantial list of known numerical identities extended to the GFP. We also adapt some known identities given
for Fibonacci or Lucas polynomials to the GFP.

Bergum and Hoggatt \cite{bergum:hoggatt} gave a list of more than twenty identities for their definition of generalized
Fibonacci polynomials. Koshy \cite{koshypartII}  has also a big collection of identities for Fibonacci polynomials and Lucas polynomials.
We generalize or adapt many of the identities given by either Bergum et al.  \cite{bergum:hoggatt}, Wu et al.  \cite{wu}, or Koshy \cite{koshypartII}
to our definition of the GFP. Most of the numerical identities for Fibonacci and Lucas numbers that we extend to GFP can be found
in the books, articles or webpages on Fibonacci and Lucas numbers and their applications, \cite{hansen, koshy, melham, morgado, surrey, Vajda}.

We note that once we have an identity --from literature-- for numerical sequences it is not too complicated to extend this to GFPs.
However,  the numerical  identities do not  extend automatically to GFPs; they  need some adjustments (some of the identities in the paper
were balanced using \emph{Mathematica}$^{\text{\textregistered}}$).
Therefore, the aim of this paper is to give a collection of identities for the GFP. In particular, the identities that we give here
apply to the following familiar polynomial sequences: Fibonacci polynomials, Lucas polynomials, Pell polynomials, Pell-Lucas polynomials,
Fermat polynomials, Fermat-Lucas polynomials, Chebyshev first kind polynomials, Chebyshev second kind polynomials,
Jacobsthal polynomials, Jacobsthal-Lucas polynomials, and Morgan-Voyce  polynomials.

\section{Generalized Fibonacci polynomials}\label{GFP}

In this section we introduce the generalized Fibonacci polynomial sequences. This definition gives rise to some known
polynomial sequences which are mentioned below. The definition matches the definitions of polynomial sequences given
in other papers on this topic, for example the definition given by Fl\'{o}rez et al., Hoggatt et al., and
Koshy \cite{florezHiguitaMukCharact, HoggattLong, koshy}, respectively.

For the remaining part of this section we reproduce the definitions by Fl\'orez et al. \cite{florezHiguitaMukstar,florezHiguitaMukCharact,florezHiguitaMuk},
for generalized Fibonacci polynomials.  We now give the two second order polynomial recurrence relations in which we divide the GFP.

\begin{equation}\label{Fibonacci;general:FT}
\Ft{0}(x)=0, \; \Ft{1}(x)= 1,\;  \text{and} \;  \Ft{n}(x)= d(x) \Ft{n - 1}(x) + g(x) \Ft{n - 2}(x) \text{ for } n\ge 2
\end{equation}
where $d(x)$, and $g(x)$ are fixed non-zero polynomials in $\mathbb{Q}[x]$.

We say a polynomial recurrence relation is of \emph{Fibonacci type} if it satisfies the relation given in \eqref{Fibonacci;general:FT}, and of \emph{Lucas type} if:

\begin{equation}\label{Fibonacci;general:LT}
\Lt{0}(x)=p_{0}, \; \Lt{1}(x)= p_{1}(x),\;  \text{and} \;  \Lt{n}(x)= d(x) \Lt{n - 1}(x) + g(x) \Lt{n - 2}(x) \text{ for } n\ge 2
\end{equation}
where $|p_{0}|=1$ or $2$ and $p_{1}(x)$, $d(x)=\alpha p_{1}(x)$, and $g(x)$ are fixed non-zero  polynomials in $\mathbb{Q}[x]$ with
$\alpha$ an integer of the form $2/p_{0}$.

To use similar notation for \eqref{Fibonacci;general:FT} and \eqref{Fibonacci;general:LT} on certain occasions we write
$p_{0}=0$, $p_{1}(x)=1$ to indicate the initial conditions of Fibonacci type polynomials. Some known examples
of Fibonacci type polynomials and of Lucas type polynomials are in Table \ref{equivalent}
(see also \cite{florezHiguitaMukCharact,florezHiguitaMuk,Pell,Fermat,koshy}).

Suppose that $G_{n}$ is either $\Ft{n}(x)$ or $\Lt{n}(x)$ for all $n\ge 0$ and $d^2(x)+4g(x)> 0$,  then the explicit formula for the
recurrence relations in
 \eqref{Fibonacci;general:FT} and \eqref{Fibonacci;general:LT}  is given by
\begin{equation*}\label{solutionrecurrencerelationuno}
 G_{n}(x) = t_1 a^{n}(x) + t_2 b^{n}(x)
\end{equation*}
where $a(x)$ and $b(x)$ are the solutions of the quadratic equation associated with the second order
recurrence relation $G_{n}(x)$. That is,  $a(x)$ and $b(x)$ are the solutions of $z^2-d(x)z-g(x)=0$.
If $\alpha=2/p_{0}$, then the Binet formula for Fibonacci type polynomials is stated in  \eqref{bineformulauno} and the Binet formula
for Lucas type polynomials is stated in \eqref{bineformulados}
(for details on the construction of the two Binet formulas see \cite{florezHiguitaMukCharact}).

\begin{equation}\label{bineformulauno}
\Ft{n}(x) = \dfrac{a^{n}(x)-b^{n}(x)}{a(x)-b(x)}
\end{equation}
and
\begin{equation}\label{bineformulados}
\Lt{n}(x)=\dfrac{a^{n}(x)+b^{n}(x)}{\alpha}.
\end{equation}

Note that for both types of sequences: $$a(x)+b(x)=d(x), \quad a(x)b(x)= -g(x), \quad \text{ and } \quad a(x)-b(x)=\sqrt{d^2(x)+4g(x)}$$
where $d(x)$ and $g(x)$ are the polynomials defined in \eqref{Fibonacci;general:FT} and \eqref{Fibonacci;general:LT}.

A sequence of Lucas type (Fibonacci type) is \emph{equivalent} or \emph{conjugate} to a sequence of Fibonacci type (Lucas type),
if their recursive sequences are determined by the same polynomials $d(x)$ and $g(x)$. Notice that two equivalent polynomials
have the same $a(x)$ and $b(x)$ in their Binet representations. Examples of equivalent polynomials are given in Table \ref{equivalent}.
Note that the leftmost polynomials in Table \ref{equivalent} are of Lucas type and their equivalent Fibonacci type polynomials
are in the second column on the same line. Table \ref{familiarfibonacci} shows some familiar examples of those types of polynomial sequences.

\begin{table} [!ht]
\begin{center}\scalebox{0.8}{
\begin{tabular}{|l|l|l|l|} \hline
  Polynomial            & Initial value     & Initial value	& Recursive Formula 						       \\	
    			 &$G_0(x)=p_0(x)$  	&$G_1(x)=p_1(x)$	&$G_{n}(x)= d(x) G_{n - 1}(x) + g(x) G_{n - 2}(x)$ 	   \\  \hline   \hline
  Fibonacci             	 & $0$	    &$1$      	&$F_{n}(x) = x F_{n - 1}(x) + F_{n - 2}(x)$	 	       \\
  Lucas 	             	 &$2$	    & $x$ 	 	&$D_n(x)= x D_{n - 1}(x) + D_{n - 2}(x)$                \\ 						
  Pell			    		 &$0$	    & $1$       &$P_n(x)= 2x P_{n - 1}(x) + P_{n - 2}(x)$               \\
  Pell-Lucas 	    		 &$2$	    &$2x$       &$Q_n(x)= 2x Q_{n - 1}(x) + Q_{n - 2}(x)$               \\
  Pell-Lucas-prime 	    	 &$1$	    &$x$       	&$Q_n^{\prime}(x)= 2x Q_{n - 1}^{\prime}(x) + Q_{n - 2}^{\prime}(x)$               \\
  Fermat  	                 &$0$	    & $1$      	&$\Phi_n(x)= 3x\Phi_{n-1}(x)-2\Phi_{n-2}(x) $           \\
  Fermat-Lucas	             &$2$	    &$3x$  		&$\vartheta_n(x)=3x\vartheta_{n-1}(x)-2\vartheta_{n-2}(x)$\\
  Chebyshev second kind      &$0$	    &$1$       	&$U_n(x)= 2x U_{n-1}(x)-U_{n-2}(x)$  	 	              \\
  Chebyshev first kind       &$1$	    &$x$       	&$T_n(x)= 2x T_{n-1}(x)-T_{n-2}(x)$  \\	 	
  Jacobsthal  	  		     &$0$	    &$1$        &$J_n(x)= J_{n-1}(x)+2xJ_{n-2}(x)$ 	 	                  \\
  Jacobsthal-Lucas	         &$2$		&$1$        &$j_n(x)=	j_{n-1}(x)+2xj_{n-2}(x)$                       \\
  Morgan-Voyce	             &$0$		&$1$      	&$B_n(x)= (x+2) B_{n-1}(x)-B_{n-2}(x) $  	 	         \\
  Morgan-Voyce 	             &$2$		&$x+2$      &$C_n(x)= (x+2) C_{n-1}(x)-C_{n-2}(x)$  	 	         \\
  Vieta 		             &$0$ 	   	&$1$	    &$V_n(x)=x V_{n-1}(x)-V_{n-2}(x)$ 	    \\
  Vieta-Lucas 		         &$2$ 	   	&$x$	    &$v_n(x)=x v_{n-1}(x)-v_{n-2}(x)$      \\
  \hline
\end{tabular}}
\end{center}
\caption{Recurrence relation of some GFP.} \label{familiarfibonacci}
\end{table}

\begin{table} [!ht]
\begin{center}\scalebox{0.8}{
\begin{tabular}{|l|l|c|c|c|l|l|} \hline
  Polynomial  of	    & Polynomial of&$\alpha$ & $d(x)$& $g(x)$ &$a(x)$ 	              & $b(x)$\\	
 First type 	 	&  Second type  &	&	 & 	&      					  &	\\
$ \Lt{n}(x)$	 		&  $ \Ft{n}(x)$  	&	&	 & 	&      					  &	\\ \hline \hline
  $D_n(x)$     		&  $F_n(x)$ 	& $1$ & $x$   & $1$  & $(x+\sqrt{x^2+4})/2$    & $(x-\sqrt{x^2+4})/2$   \\ 						
  $Q_n^{\prime}(x)$	&  $P_n(x)$     & $1$ & $2x$  & $1$  & $x+\sqrt{x^2+1}$	       & $x-\sqrt{x^2+1}$        \\
  $\vartheta_n(x)$	&  $\Phi_n(x)$  & $1$ & $3x$  & $-2$ & $(3x+\sqrt{9x^2-8})/2$  & $(3x-\sqrt{9x^2-8})/ 2$ \\
  $T_n(x)$  		&  $U_n(x)$		& $2$ & $2x$  & $-1$ & $x +\sqrt{x^2-1}$       & $x -\sqrt{x^2-1}$       \\
  $j_n(x)$	   		&  $J_n(x)$     & $1$ & $1$   & $2x$ & $(1+\sqrt{1+8x})/2$     & $(1-\sqrt{1+8x})/2$      \\
  $C_n(x)$ 	   		&  $B_n(x)$ 	& $1$ & $x+2$ & $-1$ & $(x+2+\sqrt{x^2+4x})/2$ & $(x+2-\sqrt{x^2+4x})/2$  \\
  $v_n(x)$ 	   		&  $V_n(x)$     & $1$ &  $x$   & $-1$ & $(x+\sqrt{x^2-4})/2$    & $(x-\sqrt{x^2-4})/2$     \\ \hline
\end{tabular}}
\end{center}
\caption{Binet formulas for Lucas type $L_n(x)$ and its equivalent Fibonacci type $R_n(x)$.} \label{equivalent}
\end{table}

\textbf{Note.}  The definition of the Generalized Fibonacci Polynomial by Fl\'orez et al. \cite{florezHiguitaMuk},
differs from the definition in this paper due to the initial conditions of the Fibonacci type
polynomials. Thus, the initial conditions for the Fibonacci type polynomials given by Fl\'orez, Higuita, and Mukherjee
\cite{florezHiguitaMuk} are  $G_{0}(x)=p_{0}(x)=1$ and so implicitly $G_{-1}(x)=0$.
However, our definition for the Lucas type polynomial is identical to that found in the same article.

\section{Identities}

In this section we give a collection of identities for the GFP. These identities apply, in particular,
to our familiar set of GFP. Thus, the identities apply to: Fibonacci polynomials, Lucas polynomials,
Pell polynomials, Pell-Lucas polynomials,  Fermat polynomials, Fermat-Lucas polynomials,
Chebyshev first kind  polynomials, Chebyshev second kind polynomials,  Jacobsthal polynomials,
Jacobsthal-Lucas  polynomials and Morgan-Voyce polynomials. The identities below and their proofs
are expressed in terms of  $\alpha$, $d(x)$,  $g(x)$, $a(x)$, and  $b(x)$. Table \ref{equivalent} gives
those values for the  polynomials mentioned above.
If one is interested in identities for the special case of Fibonacci polynomials and Lucas polynomials, they may
see a larger collection in Koshy \cite{koshypartII}.

For example, suppose we  want to apply the identity in Proposition \ref{identities:from:Fibo}  part (1)
which is  $(a(x)-b(x))^2 \Ft{n}(x)  =\alpha( \Lt{n+1} (x)+g \Lt{n-1}(x))$, to Chebyshev polynomials
of the first and second kind, namely $T_n(x)$ and $U_n(x)$ respectively.  We begin by getting the
appropriated information for $U_n(x)$ and $T_n(x)$ given in Table \ref{equivalent} line 4. Thus,
$a(x) =x +\sqrt{x^2-1}$, $b(x)=x -\sqrt{x^2-1}$, $\alpha=2$, and $g(x)=-1$. Next, we observe that $U_n(x)$
is Fibonacci type, so $ \Ft{n}(x) $ equals $U_n(x)$ and from Table \ref{equivalent} line 4, we get that
its equivalent Lucas type polynomial $ \Lt{n} (x)$  equals $T_n(x)$. Note that $(a(x)-b(x))^2=4x^2-4$. Now substituting
all this information into the identity given in Proposition \ref{identities:from:Fibo}  part (1) we obtain,
\[(4x^2-4)U_n(x)=2\left(T_{n+1}(x)+(-1)T_{n-1}(x) \right), \text{ for all } n.\]
If we want to apply the same identity for Jacobsthal $J_n(x)$, then we substitute the information
given in  Table \ref{equivalent} line 5 into Proposition \ref{identities:from:Fibo} part (1).
So, we have
\[(1+8x)J_n(x)=j_{n+1}(x)+(2x)j_{n-1}(x), \text{ for all } n.\]

The identities
(\ref{I1})--(\ref{I3}), (\ref{I5})--(\ref{I15}), (\ref{I42}), (\ref{I49})--(\ref{I51}),
(\ref{I56}), (\ref{I57}), (\ref{I63}), (\ref{I65}), (\ref{I67})--(\ref{I69})
are generalizations of numerical identities by Vajda \cite{Vajda}. The identities
(\ref{I16})--(\ref{I39}), (\ref{I54}), (\ref{I55}), (\ref{I64}), (\ref{I82})--(\ref{I94})
are generalizations of numerical identities provided by Koshy \cite{koshy}.  The identities
(\ref{I4}), (\ref{I40}, (\ref{I41}), (\ref{I44})--(\ref{I46}), (\ref{I52}), (\ref{I53}),
(\ref{I58})--(\ref{I62}), (\ref{I66}), (\ref{I70})--(\ref{I81})
are generalizations of numerical identities found in the webpage \cite{surrey}. The identities
(\ref{I43}), (\ref{I47}), (\ref{I48})
are generalizations of numerical identities by Benjamin and Quinn \cite{benjamin}. The identities
(\ref{I95})--(\ref{I100}) are generalizations of numerical identities present in \cite{KoshyGao}.

\textbf{Note}: For the sake of simplicity throughout the rest of this paper, we use $\Ft{0}$, $\Lt{0}$, $d$, $g$, $a$, and $b$
instead of $\Ft{0}(x)$, $\Lt{0}(x)$, $d(x)$,  $g(x)$, $a(x)$, and  $b(x)$, respectively.

The expression $(a-b)^2$ appears very often in the following list of identities. It is easy to see that
$$(a-b)^2= \alpha  \Lt{2} +2g,\quad (a-b)^2=   \Ft{3} +3g \quad \text{and} \quad (a-b)^2=d^2+4g.$$
However, in the following list of identities we keep the factor $(a-b)^2$ as is.

Proposition \ref{Cassini_Catalan} parts \eqref{Cassini} and \eqref{Catalan} are generalizations of
Cassini and Catalan identities respectively.  One can refer to the book by Koshy \cite{koshy} to learn more about them.
We have found that all of the identities in this paper can be proved by substituting the Binet formula on
each side of the proposed identity. Since the reader can check the proofs
without major difficulty, we prove some identities and leave the others as exercises.

\begin{proposition}[\cite{florezHiguitaMukCharact}] \label{divisity:Hogat:property1}
If $\{  \Lt{n} \}$ and $\{  \Ft{n} \}$ are equivalent generalized Fibonacci polynomial sequences, then

\begin{enumerate}[(1)]
  \item $ \Ft{m+n+1} =  \Ft{m+1}  \Ft{n+1} +g \Ft{m}  \Ft{n} $
  \item if $n\ge m$, then $ \Ft{n+m} = \alpha  \Ft{n}   \Lt{m} -(-g)^{m}  \Ft{n-m} $
  \item if $n\ge m$, then $ \Ft{n+m} = \alpha  \Ft{m}  \Lt{n}  +(-g)^{m}  \Ft{n-m} $
  \item $(a-b)^2 \Ft{m+n+1}  = \alpha^2  \Lt{m+1}  \Lt{n+1} +\alpha^2 g \Lt{m}  \Lt{n} .$
\end{enumerate}
\end{proposition}

\begin{proposition}\label{Cassini_Catalan}If $\{  \Lt{n} \}$ and $\{  \Ft{n} \}$ are equivalent generalized
Fibonacci polynomial sequences, then
\begin{enumerate}[(1)]
 \item  \label{Cassini}   $ \Ft{n+1}  \Ft{n-1} -\Ft{n}^{2}=(-1)^{n}g^{n-1}$
\item \label{Catalan}  if  $n\ge m$, then
 $ \Ft{n}^{2}-(-g)^{n-m}\Ft{m}^{2}= \Ft{n+m}  \Ft{n-m} .$
 \end{enumerate}
\end{proposition}

\begin{proof}[Proof of part (\ref{Cassini})]

We prove this part using Binet formula \eqref{bineformulauno} from Section \ref{GFP}.  Therefore, substituting formula \eqref{bineformulauno} in
$ \Ft{n+1}  \Ft{n-1} -( \Ft{n} )^{2}$, we obtain,
$$\left[(a^{n+1}-b^{n+1})(a^{n-1}-b^{n-1})-(a^{n}-b^{n})^{2}\right]/(a-b)^{2}.$$
Simplifying this expression, we obtain
$$\dfrac{2(ab)^{n}-(a^{n+1}b^{n-1}+b^{n+1}a^{n-1})}{(a-b)^{2}}.$$
Factoring and simplifying we see that the last
expression reduces to $-(ab)^{n-1}$, and this is equal to $(-1)^{n}g^{n-1}$.
\end{proof}

\begin{proof}[Proof of part (\ref{Catalan})]
Using Binet formula \eqref{bineformulauno} in
$\Ft{n}^{2}-(-g)^{n-m}\Ft{m}^{2}$, we have, $$\left[(a^{n}-b^{n})^{2}-(ab)^{n-m}(a^{m}-b^{m})^{2}\right]/(a-b)^{2}.$$
Simplifying the right hand side of the last expression we obtain, $$\left[a^{2n}+b^{2n}-a^{n+m}b^{n-m}-a^{n-m}b^{n+m}\right]/(a-b)^{2}.$$
Factoring the last expression, we get $\left[(a^{n+m}-b^{n+m})(a^{n-m}-b^{n-m})\right]/(a-b)^{2}$, and this is equal to $ \Ft{n+m}  \Ft{n-m} $.
\end{proof}

\begin{proposition}\label{identities:from:Fibo}

Let $\{  \Lt{n} \}$ and $\{  \Ft{n} \}$ be equivalent generalized Fibonacci polynomial sequences. If $m$ and $n$ are positive integers, then

\begin{enumerate}[(1)]
\item  \label{I1} $(a-b)^2 \Ft{n}  =\alpha( \Lt{n+1} +g \Lt{n-1} )$

\item  \label{I2} $g \Ft{n-1} + \Ft{n+1} =\alpha  \Lt{n} $

\item  \label{I3} $ \Ft{n+2} -g^2 \Ft{n-2} =\alpha^2 \Lt{1}  \Lt{n} $

\item  \label{I4} $ \Ft{n+2} +g^{2} \Ft{n-2} =(d^2+2g) \Ft{n} $

\item \label{I5} $ \Ft{n+2} +g^2 \Ft{n-2} =\alpha  \Lt{2}  \Ft{n} $

\item  \label{I6}   $\alpha  \Lt{n}  \Ft{n} = \Ft{2n} $

\item  \label{I7}   $\alpha ( \Ft{n+1}  \Lt{n+1} - \Ft{n}  \Lt{n} )= \Ft{2n+2} - \Ft{2n} $

\item  \label{I8}   $\alpha( \Lt{m}  \Ft{n} + \Lt{n}  \Ft{m} )= 2 \Ft{n+m} $

\item  \label{I9}   if $n\ge m$, then $\alpha( \Lt{m}  \Ft{n} - \Lt{n}  \Ft{m} )= 2(-g)^m  \Ft{n-m} $

\item  \label{I10}   if $\alpha=1$, then $2  \Lt{2n} -\Lt{n}^2=(a-b)^2\Ft{n}^2$

\item  \label{I11}
$   \Lt{2n} -2(-g)^n= \left(((a-b) \Ft{n} )^2-2(\alpha-1)(-g)^n\right)/\alpha $

  \item  \label{I12}   if $\alpha=1$, then $(a-b)^2\Ft{n}^2- \Lt{n}^2=4(-1)^{n+1}g^n$

  \item  \label{I13}

   $\alpha^2\left(g { \Lt{n}}^{2}+ \Lt{n+1}^{2}\right)=(a-b)^2\left(g \Ft{n}^{2}+\Ft{n+1}^{2}\right)=(a-b)^{2} \Ft{2n+1} $

   \item  \label{I14}   $ \Lt{2}  \Ft{n} +\alpha  \Lt{1}  \Lt{n} =2 \Ft{n+2} /\alpha$

 \item  \label{I15}   $ (a-b)^2  \Lt{1}  \Ft{n} +\alpha  \Lt{2}  \Lt{n} = 2 \Lt{n+2} $

 \item  \label{I16}   $\alpha  \Ft{n+1}  \Lt{n} = \Ft{2n+1} +(-g)^{n}$

 \item  \label{I17}   if $n\ne m$, then
 $ \Lt{m}  \Ft{n-m+1} +g \Lt{m-1}  \Ft{n-m} = \Lt{n} $

 \item  \label{I18}   if  $n\ge m$, then
 $\alpha^2  \Lt{m+n}^{2}+g^{2n}(a-b)^2\Ft{m-n}^{2}=\alpha^2  \Lt{2n}  \Lt{2m} $

\item  \label{I19}   if $m\ge n$, then
 $\alpha^2 g^{2n} \Lt{m-n} ^{2}+(a-b)^2 \Ft{m+n}^{2}=\alpha^2  \Lt{2n}  \Lt{2m} $

 \item  \label{I20}   if  $m\ge n$, then
 $\alpha^{2}\left(\Lt{2m+2n}^{2}-g^{4n}\Lt{2m-2n}^{2}\right)=(a-b)^2  \Ft{4m}  \Ft{4n} $

\item  \label{I21}   $(a-b)^2 \Ft{2n}^{2}+2g^{2n}=\alpha \Lt{4n} $

\item  \label{I22}   $(a-b)^2\Ft{2n+1}^{2}-2(g)^{2n+1}=\alpha{ \Lt{4n+2}} $

\item  \label{I23}    $\alpha{ \Lt{n}} { \Ft{2n}} -(-g)^{n}{ \Ft{n}} ={ \Ft{3n}} $

\item  \label{I24}   ${ \Ft{n}} \left(\alpha{ \Lt{2n}} +(-g)^{n}\right)={ \Ft{3n}} $

\item  \label{I25}   ${ \Ft{4n+1}} -g^{2n}=\alpha { \Lt{2n+1}} { \Ft{2n}} $

\item  \label{I26}   ${ \Ft{4n+3}} +(-g)^{2n+1}=\alpha { \Lt{2n+1}} { \Ft{2n+2}} $

\item  \label{I27}   $\alpha { \Lt{n+1}} -(a-b)^{2}{ \Ft{n}} =\alpha (-g){ \Lt{n-1}} $

\item  \label{I28}   $\alpha\left({ \Ft{n+1}} { \Lt{n+2}} -d{ \Ft{n+2}} { \Lt{n}} \right)= g{ \Ft{2n+1}} -(-g)^{n}(d^2-g)$

\item  \label{I29}  if  $n\ge m$, then
$(a-b)^{2}\left(\Ft{n+m}^{2}+g^{2m}\Ft{n-m}^{2}\right)  =\alpha^{2}{ \Lt{2n}} { \Lt{2m}} -4(-g)^{n+m}$

\item  \label{I30}   if $n\ge m$, then
\[(a-b)^{2}\left({ \Ft{n+m}} { \Ft{n+m+1}} +g^{2m}{ \Ft{n-m}} { \Ft{n-m+1}} \right) =\alpha^{2}{ \Lt{2n+1}} { \Lt{2m}} -2(-g)^{n+m}d\]

\item  \label{I31}   ${ \Ft{n+4}} \Ft{n+1}^{2}-{ \Ft{n}} \Ft{n+3}^{2}=\alpha d(-g)^{n}{ \Lt{n+2}} $

 \item  \label{I32}   $\alpha\left({ \Ft{n+4}} \Lt{n+1}^{2}-{ \Ft{n}} \Lt{n+3}^{2}\right)= d^{3}(-g)^{n}{ \Lt{n+2}} $

 \item  \label{I33}   $2\alpha  \Lt{m+n} =\alpha^2{ \Lt{m}} { \Lt{n}} +(a-b)^2{ \Ft{m}} { \Ft{n}} $

 \item  \label{I34}   $\alpha( \Lt{m+n} -(-g)^{n}{ \Lt{m-n}} )=(a-b)^2{ \Ft{m}} { \Ft{n}} $

 \item  \label{I35}   $\alpha^{2}{ \Lt{2m+1}} { \Lt{2n+1}} =\alpha^2 \Lt{m+n+1}^{2}-(a-b)^2(g)^{2n+1}\Ft{m-n}^{2}$

 \item  \label{I36}   if $n\ge m$, then
 \[\alpha (a-b) { \Ft{n}} { \Lt{n+m}} -\alpha^{2}{ \Lt{n}} { \Lt{n-m}}
 \]
 equals
 \[ (a-b)({ \Ft{2n+m}} -(-g)^{n} { \Ft{m}} )-\alpha\left( \Lt{2n-m} -(-g)^{n-m}{ \Lt{n}} \right)\]

 \item  \label{I37}   $\alpha^{2}{ \Lt{n-1}} { \Lt{n+1}} -(a-b)^2({ \Ft{n}} )^{2}=(-g)^{n-1}\left(\alpha{ \Lt{2}} -2g\right)$

 \item  \label{I38}   $(a-b)^2{ \Ft{2n+3}} { \Ft{2n-3}} =\alpha\left( \Lt{4n} -(-g)^{2n-3}{ \Lt{6}} \right)$

 \item  \label{I39}   ${ \Lt{5n}} ={ \Lt{n}} \left[(\alpha{ \Lt{2n}} -(-g)^{n})^{2}+(a-b)^{2}(-g)^{n} \Ft{n}^{2}\right]$

 \item  \label{I40}   $ \Ft{n+5} -g^{2} \Ft{n+1} =d\alpha  \Lt{n+3} $

 \item  \label{I41}   $ \Ft{n+5} +g^{2} \Ft{n+1} =\alpha  \Lt{2}   \Ft{n+3} $

 \item  \label{I42}   $d\Ft{n}^{2}+2g{ \Ft{n-1}} { \Ft{n+1}} ={ \Ft{2n}} $

 \item  \label{I43}   $\Ft{n+1}^{2}-g^{2}\Ft{n-1}^{2}=d \Ft{2n} $

 \item  \label{I44}   $\Ft{n+3}^{2}+g^{3}\Ft{n}^{2}= \Ft{2n+3}  \Ft{3} $

 \item  \label{I45}   if  $n\ge m$, then
 $ \Ft{n+m+1}^{2}+g^{2m+1}\Ft{n-m}^{2}= \Ft{2n+1}  \Ft{2m+1} $

 \item  \label{I46}   if  $n\ge m$, then
 $ \Ft{n+m}^{2}-g^{2m}\Ft{n-m}^{2}= \Ft{2n}  \Ft{2m} $

 \item  \label{I47}   if  $n\ge m$, then
 $ \Ft{n}  \Ft{m+1} - \Ft{m}  \Ft{n+1} =(-g)^{m} \Ft{n-m} $

 \item  \label{I48}   if  $n\ge m$, then
 $ \Ft{n+1}  \Ft{m+1} -g^{2} \Ft{n-1}  \Ft{m-1} =d \Ft{n+m} $

 \item  \label{I49}   if  $n\ge m$, then
 $ \Ft{n+1}  \Ft{m} +g  \Ft{n}  \Ft{m-1} = \Ft{n+m} $

 \item  \label{I50}   if $n\ge m$, then
 $ \Ft{n-m+1}  \Ft{m} +g  \Ft{n-m}  \Ft{m-1} = \Ft{n} $

 \item  \label{I51}   $ \Ft{n}  \Ft{n+1} - \Ft{n-1}  \Ft{n+2} =d(-g)^{n-1}$

 \item  \label{I52}  if $i\ge 0$, then
 $ \Ft{n+i}  \Ft{n+m} -  \Ft{n}  \Ft{n+m+i} =(-g)^{n} \Ft{i}  \Ft{m} $

 \item  \label{I53}   $\sqrt{(a-b)^{2}\Ft{n}^{2}+4(-g)^{n}}-g \Ft{n-1} = \Ft{n+1} $

 \item  \label{I54}   if  $r\le \min\{ m, n, s, t\}$ and $m+n=s+t$, then
 $$ \Ft{m}  \Ft{n} -  \Ft{s}  \Ft{t} =(-g)^{r}\left( \Ft{m-r}  \Ft{n-r} - \Ft{s-r}  \Ft{t-r} \right)$$

 \item  \label{I55}    $ \Ft{n+2}  \Ft{m+1} - g^{2} \Ft{n}  \Ft{m-1} =d \Ft{n+m+1} $

 \item  \label{I56}   $ \Lt{2n} \left[\alpha^{2}( \Lt{4n} -g^{2n}/\alpha)^{2}+g^{2n}(a-b)^{2}\Ft{2n}^{2}\right]= \Lt{10n} $

 \item  \label{I57}   $ \Lt{2n} +2(-g)^{n}/\alpha=\alpha \Lt{n}^{2}$

 \item  \label{I58}  if  $n\ge m$, then
 $ \Lt{n+m} +(-g)^{m} \Lt{n-m} =\alpha  \Lt{m}  \Lt{n} $

 \item  \label{I59}   $(a-b)\sqrt{\Lt{n}^{2}-4(-g)^{n}/\alpha^{2}}-g \Lt{n-1} =  \Lt{n+1} $

 \item  \label{I60}   $(a-b)^{2}\Ft{n+2}^{2}+g^{2}\alpha^{2}\Lt{n}^{2}=\alpha^{2}  \Lt{2}  \Lt{2n+2} $

 \item  \label{I61}   $\alpha^{2}\left(\Lt{n+1}^{2}-g^{2} \Lt{n-1} ^{2}\right)=(a-b)^{2} \Ft{2}  \Ft{2n} $

 \item  \label{I62}   $\alpha^{2}\Lt{n+1}^{2}-(a-b)^{2}g \Ft{n}^{2}=\alpha^{2}d \Lt{2n+1} $

 \item  \label{I63}   if  $n\ge m$, then $ \Lt{n+m} +(-g)^{m} \Lt{n-m} =\alpha  \Lt{n}  \Lt{m} $

 \item  \label{I64}   $ \Ft{n+1}  \Lt{n+1} +g \Ft{n}  \Lt{n} = \Lt{2n+1} $

 \item  \label{I65}   if  $m\ge n$, then
 $ \Ft{m}  \Ft{2m}  \Ft{3n} =\Ft{m+n}^{3}-(-g)^{3n}\Ft{m-n}^{3}-\alpha (-g)^{m}\Ft{n}^{3} \Lt{m} $

 \item  \label{I66}   $(a-b)^{2}\left[ \Ft{n}^{2}+\Ft{n+1}^{2}\right]=\alpha^{2}\left[\Lt{n}^{2}+\Lt{n+1}^{2}\right]+4(-g)^{n}(g-1)$

 \item  \label{I67}   $ \Ft{m}  \Lt{n} +g \Ft{m-1}  \Lt{n-1} = \Lt{m+n-1} $

 \item  \label{I68}   $ \Ft{n+m+i}  \Lt{n} - \Ft{n+m}  \Lt{n+i} =(-g)^{n} \Lt{m}  \Ft{i} $

 \item  \label{I69}   $ \Ft{n+i}  \Lt{n+m} - \Ft{n}  \Lt{n+m+i} =(-g)^{n} \Lt{m}  \Ft{i} $

 \item  \label{I70}   $\alpha^{2}( \Lt{n+m+i}  \Lt{n} - \Lt{n+m}  \Lt{n+i} )=(-g)^{n}(a-b)^{2} \Ft{m}  \Ft{i} $

 \item  \label{I71}   $(-g)^{k} \Ft{n}  \Ft{m-k} +(-g)^{m} \Ft{n-m}  \Ft{k} +(-g)^{n} \Ft{m}  \Ft{k-n} =0$

 \item  \label{I72}   $(-g)^{k} \Lt{n}  \Ft{m-k} +(-g)^{m} \Lt{k}  \Ft{n-m} +(-g)^{n} \Lt{m}  \Ft{k-n} =0$

 \item  \label{I73}   $(a-b)^{2} \Ft{jk+r}  \Ft{ju+v} =\alpha\left[ \Lt{j(k+u)+r+v} -(-g)^{ju+v} \Lt{j(k-u)+r-v} \right]$

 \item  \label{I74}   $\alpha  \Ft{jk+r}  \Lt{ju+v} =\left[ \Ft{j(k+u)+r+v} +(-g)^{ju+v} \Ft{j(k-u)+r-v} \right]$

 \item  \label{I75}   $\alpha  \Lt{jk+r}  \Lt{ju+v} =\left[ \Lt{j(k+u)+r+v} +(-g)^{ju+v} \Lt{j(k-u)+r-v} \right]$

 \item  \label{I76}   if $m+n=s+t$, then $\Ft{m}  \Lt{n} - \Ft{s}  \Lt{t} =(-g)^{r}\left[ \Ft{m-r}  \Lt{n-r} - \Ft{s-r}  \Lt{t-r} \right]$

 \item  \label{I77}   if $m+n=s+t$, then

 \[(a-b)^{2} \Ft{m}  \Ft{n} -\alpha^{2} \Ft{s}  \Lt{t} =(-g)^{r}\left[(a-b)^{2} \Ft{m-r}  \Ft{n-r} -\alpha^{2} \Lt{s-r}  \Lt{t-r} \right]\]

 \item \label{I78} $\Ft{n+1}^{3}+gd \Ft{n}^{3}-g^{3}\Ft{n-1}^{3}=d \Ft{3n} $

 \item \label{I79} $\alpha^{2}\left[ \Lt{n+1}^{3}+gd \Lt{n}^{3}-g^{3}\Lt{n-1}^{3}\right]=d(a-b)^{2} \Lt{3n} $

\item \label{I80} $\alpha^{2}\left[ \Ft{n+1} \Lt{n+1}^{2}+gd  \Ft{n} \Lt{n}^{2}-g^{3} \Ft{n-1} \Lt{n-1}^{2}\right]=d(a-b)^{2} \Ft{3n} $

\item \label{I81} $ \Lt{n+1}  \Ft{n+1}^{2}+gd  \Lt{n} \Ft{n}^{2}-g^{3} \Lt{n-1} \Ft{n-1}^{2}=d \Lt{3n} $

\item \label{I82} $ \Ft{3n} =\alpha  \Lt{3}  \Ft{3n-3} +g^{3} \Ft{3n-6} $

\item \label{I83} $ \Ft{3n} =(a-b)^{2}\Ft{n}^{3}+3(-g)^{n} \Ft{n} $

\item \label{I84} $ \Ft{r+s+t} = \Ft{r+1}  \Ft{s+1}  \Ft{t+1} +gd \Ft{r}  \Ft{s}  \Ft{t} -g^{3} \Ft{r-1}  \Ft{s-1}  \Ft{t-1} $

\item \label{I85} if $k< s <m$ are non-negative integers, then
$$ \Ft{m+k}  \Ft{m-k} - \Ft{m+s}  \Ft{m-s} =(-g)^{m-s} \Ft{s-k}  \Ft{s+k}$$

\item \label{I86} if $r<n$ is a positive integer, then
$ \Ft{r}  \Ft{m+n} = \Ft{m+r}  \Ft{n} -(-g)^{r} \Ft{n-r}  \Ft{m} $

\item \label{I87} if $r<m$ is a positive integer, then
$$d \Ft{2r}\left[(a-b)^2 \Ft{m}^{2}+2(-g)^{m}\right]$$
equals
$$\Ft{m+r+1}^{2}-g^{2}\Ft{m+r-1}^{2}-g^{2r}\Ft{m-r+1}^{2}+g^{2r+2}\Ft{m-r-1}^{2}$$

\item \label{I88} if $r<m$ is a positive integer, then
$$\alpha^{2} \Lt{1}  \Lt{2m}  \Ft{2r} =
\Ft{m+r+1}^{2}-g^{2}\Ft{m+r-1}^{2}-g^{2r}\Ft{m-r+1}^{2}+g^{2r+2}\Ft{m-r-1}^{2}$$

\item \label{I89}  $\left(d^2 \Ft{n+1}^{2}+\Ft{n+2}^{2}\right)^{2}=
\Ft{n}^{2}\left( \Ft{n+4} -d^{2} \Ft{n+2} \right)^{2}+(2d \Ft{n+1}  \Ft{n+2} )^{2}$

(This is an adaptation of the Pythagorean Theorem for the Fibonacci type polynomials)

\item \label{I90}
$\left(d^2 \Lt{n+1}^{2}+\Lt{n+2}^{2}\right)^{2}=\Lt{n}^{2}\left( \Lt{n+4} -d^{2} \Lt{n+2} \right)^{2}+(2d \Lt{n+1}  \Lt{n+2} )^{2}$

(This is an adaptation of the Pythagorean Theorem for the Lucas type polynomials)

\item \label{I91}
$\Ft{n+2m+1}^{2}+g^{2m+1}\Ft{n}^{2}= \Ft{2m+1}  \Ft{2n+2m+1} $

\item \label{I92}
$\alpha^{2}\left(\Ft{n+m}^{2} \Lt{n+m}^{2}-g^{2n}\Ft{m}^{2}\Lt{m}^{2}\right)= \Ft{2n}  \Ft{4m+2n} $

\item \label{I93} if $ r<m$ is a positive integer, then
$$\alpha^{2} \Lt{2m-2}  \Lt{2r} -2(-g)^{m+r-2}(2g+d^2)$$
equals
$$(a-b)^{2}\left( \Ft{m+r}  \Ft{m+r-2} +g^{2r} \Ft{m-r}  \Ft{m-r-2} \right)$$

\item \label{I94}
$\displaystyle{\left[\left(\alpha  \Lt{n} +(a-b) \Ft{n}\right)/2\right]^{m}=\left(\alpha  \Lt{mn} +(a-b) \Ft{mn} \right)/2}$

 \item \label{I95} $ \Ft{3^{n+1}}  = (a-b)^2 \Ft{3^n}^3 + 3(-g)^{3^n}\Ft{3^n}$

\item \label{I96} $ \Ft{5^{n+1}}  = (a-b)^4 \Ft{5^n}^5 + 5(a-b)^2(-g)^{5^n}\Ft{5^n}^3 + 5(-g)^{2(5^n)}\Ft{5^n}$

\item \label{I97}
$ \Ft{7^{n+1}} = A^6 \Ft{7^n} ^7 + 7A^4(-g)^{7^n} \Ft{7^n} ^5 + 14A^2(-g)^{2(7^n)}\Ft{7^n}^3 + 7(-g)^{3(7^n)} \Ft{7^n}$, where  $A=(a-b)$.

\item \label{I98} $ \Lt{2^{n+1}}  = \alpha \Lt{2^n}^2 - 2\alpha^{-1}(-g)^{2^n}$

\item \label{I99} $ \Lt{4^{n+1}}  = \alpha^3 \Lt{4^n}^4 - 4\alpha(-g)^{4^n}\Lt{4^n}^2 + 2\alpha^{-1}(-g)^{2(4^n)}$

\item \label{I100} $ \Lt{6^{n+1}}  = \alpha^5\Lt{6^n}^6 - 6\alpha^{3}(-g)^{6^n} \Lt{6^n}^4 + 9\alpha(-g)^{2(6^n)}\Lt{6^n}^2-2\alpha^{-1}(-g)^{3(6^n)}$.

\end{enumerate}
\end{proposition}

\begin{proof}[Proof of part (\ref{I1})]
	We know $ \Ft{n} $ is a Fibonacci type polynomial, so we can represent it using Binet formula \eqref{bineformulauno}. Therefore, $(a-b)^2 \Ft{n} $ can be written as $(a-b)^2[(a^n-b^n)/(a-b)]$. Simplifying, we see the above expression equals $(a^{n+1}+b^{n+1}-ab^n-ba^n)$.  This reduces to $a^{n+1}+b^{n+1}+g(a^{n-1}+b^{n-1})$, since we know that $ab=-g$.  It is easy to see that
    $$a^{n+1}+b^{n+1}+g(a^{n-1}+b^{n-1})=\alpha\left(\frac{a^{n+1}+b^{n+1}}{\alpha} + g\frac{a^{n-1}+b^{n-1}}{\alpha}\right).$$
 This and Binet formula \eqref{bineformulados} give us $\alpha( \Lt{n+1} +g \Lt{n-1} )$ and the proof is complete.
\end{proof}

\begin{proof}[Proof of part (\ref{I2})]

	Similar to part (1), we know $ \Ft{n} $ is a Fibonacci type polynomial, so from Binet formula \eqref{bineformulauno} and $g=-ab$, we have
    $$g \Ft{n-1} + \Ft{n+1} =\left(-ab(a^{n-1}-b^{n-1})+a^{n+1}-b^{n+1}\right)/(a-b).$$
    Simplifying the numerator we have $\left(a^n(a-b)+b^n(a-b)\right)/(a-b)$. This and Binet formula  \eqref{bineformulados} imply that $a^n+b^n=\alpha  \Lt{n} $.
\end{proof}

\begin{proof}[Proof of part (\ref{I3})]

Substituting  $ \Ft{n} $ with its Binet formula \eqref{bineformulauno} and $-ab$ with $g$ we obtain  $$ \Ft{n+2} -g^2 \Ft{n-2} =\left(a^{n+2}-b^{n+2}-a^2b^2(a^{n-2}-b^{n+2})\right)/(a-b).$$
 Factoring the numerator we have $\left((a^2-b^2)(a^n+b^n)\right)/(a-b)$. This implies that, $$ \Ft{n+2} -g^2 \Ft{n-2} =(a+b)(a^n+b^n).$$
 Rewriting the right hand side we obtain
    $$ \Ft{n+2} -g^2 \Ft{n-2} =\alpha^2\left(\frac{a+b}{\alpha}\right)\left(\frac{a^n+b^n}{\alpha}\right).$$
    This and Binet formula \eqref{bineformulados} for Lucas type polynomials complete the proof.
\end{proof}

\begin{proof}[Proof of part (\ref{I4})]
We use the Binet formula \eqref{bineformulauno}, $d=(a+b)$, and $g=-ab$ in the expression $[(d^{2}+2g) \Ft{n} ]$, to obtain $[(a^{2}+b^{2})(a^{n}-b^{n})]/(a-b)$. Expanding and simplifying we get $[(a^{n+2}-b^{n+2})+(ab)^{2}(a^{n-2}-b^{n-2})]/(a-b)$ which is the same as $[ \Ft{n+2} +g^{2} \Ft{n-2} ]$.
\end{proof}

\section{Acknowledgement}

The first and the last authors were partially supported by The Citadel Foundation.

\bigskip
\hrule
\bigskip

\noindent  MSC 2010:
Primary 11B39; Secondary 11B83.

\noindent \emph{Keywords: }
Fibonacci polynomials, generalized Fibonacci polynomial, Chebyshev polynomials,
Morgan-Voyce polynomials, Lucas polynomials, Pell polynomials, Fermat polynomials.

\end{document}